%% file: main.tex
\documentclass[runningheads]{llncs}

\usepackage[T1]{fontenc}
\usepackage[utf8]{inputenc}
\usepackage[british]{babel}
\usepackage{csquotes}
\usepackage{lmodern}
\usepackage{algorithm}
\usepackage{algpseudocodex}
\usepackage{amsmath}
\usepackage{amssymb}
\usepackage{mathtools}
\usepackage{braket}
\usepackage{colonequals}
\usepackage{afterpage}
\usepackage{marvosym}
\usepackage[activate={true,nocompatibility},final,tracking=true,kerning=true,spacing=true,factor=1100]{microtype}
    \AtBeginEnvironment{verbatim}{\microtypesetup{activate=false}}
    \SetProtrusion{encoding={*},family={bch},series={*},size={6,7}}{1={ ,750},2={ ,500},3={ ,500},4={ ,500},5={ ,500},6={ ,500},7={ ,600},8={ ,500},9={ ,500},0={ ,500}}
    \SetExtraKerning[unit=space]{encoding={*}, family={bch}, series={*}, size={footnotesize,small,normalsize}}{\textendash={400,400},"28={ ,150},"29={150, },\textquotedblleft={ ,150},\textquotedblright={150, }}
    \SetExtraKerning[unit=space]{encoding={*}, family={qhv}, series={b}, size={large,Large}}{1={-200,-200},\textendash={400,400}}
\usepackage{multirow}
\usepackage{pgfplots}
\usepackage{pgfplotstable}
    \pgfplotsset{compat=newest}
    \definecolor{mynavy}{RGB}{0,68,136}
    \definecolor{mysalmon}{RGB}{238,153,170}
    \definecolor{mybrown}{RGB}{153,119,0}
    \definecolor{myblue}{RGB}{102,153,204}
    \definecolor{myyellow}{RGB}{238,204,102}
    \definecolor{myred}{RGB}{153,68,85}
    \pgfplotscreateplotcyclelist{colorblindcycle}{
        color=mynavy,mark=*\\
        color=mysalmon,mark=square*\\
        color=mybrown,mark=triangle*\\
        color=black,mark=diamond*\\
        color=myblue,mark=star\\
        color=myyellow,mark=pentagon*\\
        color=myred,every mark/.append style={fill=myred!65!white},mark=otimes*\\
        color=mynavy,densely dashed,mark=*,every mark/.append style=solid\\
        color=mysalmon,densely dashed,mark=square*,every mark/.append style=solid\\
        color=mybrown,densely dashed,mark=triangle*,every mark/.append style=solid\\
        color=black,densely dashed,mark=diamond*,every mark/.append style=solid\\
        color=myblue,densely dashed,mark=star,every mark/.append style=solid\\
        color=myyellow,densely dashed,mark=pentagon*,every mark/.append style=solid\\
        color=myred,densely dashed,every mark/.append style={fill=myred!65!white},mark=otimes*,every mark/.append style=solid\\
    }
    \pgfplotsset{colorblindstyle/.style={
        minor x tick num=1,
        xtick pos=left,
        ytick pos=left,
        enlarge x limits=false,
        every x tick/.style={color=black, thin},
        every y tick/.style={color=black, thin},
        tick align=outside,
        xlabel near ticks,
        ylabel near ticks,
        cycle list name=colorblindcycle,
    }}
\usepackage{hyperref}
    
\usepackage{ellipsis}
\usepackage[capitalise]{cleveref}

\usepackage{todonotes}

\newcommand*{\bbN}{\mathbb{N}}
\newcommand*{\bbP}{\mathbb{P}}
\newcommand*{\bbR}{\mathbb{R}}
\newcommand*{\errvec}{\mathrm{err}}
\newcommand*{\hpipe}{\rotatebox[origin=c]{90}{$|$}}
\newcommand*{\abbrev}[1]{\textsc{#1}}
\newcommand*{\transp}{^{\mkern-1.5mu{\scriptscriptstyle\mathsf{T}}}}
\newcommand*{\txteq}[1]{\overset{\text{\tiny #1}}{=}}

\DeclareMathOperator{\rowsp}{rowsp}
\DeclareMathOperator{\spn}{span}

\DeclarePairedDelimiter{\abs}{\lvert}{\rvert}
\DeclarePairedDelimiter{\norm}{\lVert}{\rVert}
\DeclarePairedDelimiter{\inp}{\langle}{\rangle}


\begin{document}
    \title{The Arnoldi Aggregation for Approximate Transient Distributions of Markov Chains}
    \titlerunning{Arnoldi Aggregations for Transient Distributions of \abbrev{dtmc}s}
    \author{Patrick Sonnentag\textsuperscript{({\small\Letter})}\orcidID{0009-0002-7319-5884} \and Fabian Michel\orcidID{0009-0005-7768-9111} \and Markus Siegle\orcidID{0000-0001-7639-2280}}
    \authorrunning{P. Sonnentag et al.}
    \institute{University of the Bundeswehr Munich, Institute for Computer Engineering, Werner-Heisenberg-Weg 39, Neubiberg, 85579, Germany\\
    \email{\{patrick.sonnentag,fabian.michel,markus.siegle\}@unibw.de}}
    \maketitle
    \begin{abstract}
        The paper proposes a new aggregation method, based on the Arnoldi iteration, for computing approximate transient distributions of Markov chains.
        This aggregation is not partition-based, which means that an aggregate state may represent any portion of any original state, leading to a reduced system which is not a Markov chain.
        Results on exactness (in case the algorithm finds an invariant Krylov subspace) and minimality of the size of the Arnoldi aggregation are proven.
        For practical use, a heuristic is proposed for deciding when to stop expanding the state space once a certain accuracy has been reached.
        Apart from the theory, the paper also includes an extensive empirical section where the new aggregation algorithm is tested on several models and compared to a lumping-based state space reduction scheme.
    \end{abstract}
    \input{01-introduction}
    \input{02-preliminaries}
    \input{03-build-aggregations}
    \input{04-review}
    \input{05-conclusion}
    \bibliographystyle{splncs04}
    \bibliography{references.bib}
\end{document}

%% file: 01-introduction.tex
\section{Introduction \& Motivation}\label{sec:intro}
Approximate solutions of Markov chains by means of state space reduction have received a lot of attention in the literature.
One of the earliest works is that of Simon and Ando~\cite{simonando1961aggregation} on nearly decomposable Markov chains.
These techniques were further developed and applied to models of computer systems by Courtois~\cite{courtois1977decomposability} which also led to iterative aggregation/disaggregation methods~\cite{koury1984iterativencd}.
Most of these early works on state space aggregation focused on the stationary distribution, while Buchholz in his seminal paper on lumpability~\cite{buchholz1994lumpability} also considered transient distributions.
An overview of general transient solution methods is also given in the paper \cite{deSouzaeSilva2000transientmc}, where --- among several others --- a Krylov subspace based method is briefly described.
It is, of course, important to control the error caused by aggregation/disaggregation; see, for instance,~\cite{bucholz2014aggregation}.
Recently, state space reduction techniques for obtaining approximate transient distributions of Markov chains with formal error bounds have been studied~\cite{abate2021aggregation}.
This has been generalized in~\cite{michel2025formalbounds}, the latter work characterising exactness of Markov chain aggregations, which goes beyond the well-known concept of lumping.

To achieve a space reduction of a Markov chain with transition matrix $P$, the primary question is how to aggregate the state space, i.e.\ how to construct the reduced system.
More concretely, the question is how to choose the size and the transition matrix of the reduced system.
Secondly, once a solution for the reduced system has been obtained,
a strategy is needed for mapping this solution back onto the original chain, i.e.\ one needs to specify a disaggregation scheme, usually in the form of a disaggregation matrix $A$.
The third question concerns the choice of the initial vector for the aggregated system.

Common approaches to state space reduction use a partitioning of the original state space, leading to one aggregated state per class, which means that every original state belongs to exactly one state of the reduced system.
The reduced system is usually considered to be a Markov chain again, its transition matrix constructed as a stochastic matrix, and the initial vector of the reduced system a probability vector.
In this paper, we consider more general state space reductions, where a reduced state can represent any portion of any original state.

The Arnoldi iteration~\cite{arnoldi1951iteration} is a well-known algorithm for computing an orthonormal basis of the Krylov subspace of a vector $v$ w.r.t.\ a matrix $M$, i.e.\ the space $\spn\Set{v\transp, v\transp M, \dots, v\transp M^{j - 1}}$.
We develop the Arnoldi aggregation for Markov chains, based on the Arnoldi iteration, where the initial vector is the initial distribution ($v=p_0$) and $M=P$ is the transition matrix of the \abbrev{dtmc}.
The vectors computed as the basis of this Krylov subspace are used as the rows of the disaggregation matrix $A$, and the step matrix of the reduced system is obtained from the multiplication factors appearing in the Gram-Schmidt orthonormalization part of the Arnoldi process.
The step matrix thus obtained is not a stochastic matrix, i.e.\ the reduced system is not a Markov chain, and --- in consequence --- the vectors obtained may not be probability distributions.
We show that Arnoldi aggregations are either exact with minimal state space size or have an error of zero for the first $j-1$ time steps (where $j$ is the number of Arnoldi iterations performed)  with minimal state space size.
For transient time points larger than the number of Arnoldi iterations (which is the interesting case), the error can be estimated by a heuristic during the Arnoldi iteration, thus making it possible to dynamically control the size of the reduced system based on the desired error.

The key contributions of this paper, which is based on the thesis~\cite{sonnentag2025thesis}, are as follows:
A new aggregation scheme for Markov chains is proposed and analysed from a theoretical point of view, leading to exactness and minimality results.
In view of finite floating point precision, practically viable termination criteria for the iterative Arnoldi procedure in the aggregation setting are established.
Empirical results are provided which confirm the theoretical findings and show that the Arnoldi aggregation can be superior to other approaches.

The remainder of the paper is structured as follows:
Starting in \cref{sec:preliminaries}, we explain basic notation and formally introduce the Arnoldi iteration.
In \cref{sec:building-aggregations}, we introduce the Arnoldi aggregation and derive its main theoretical properties.
We further consider how to work with Arnoldi aggregations in floating point arithmetic, introduce the final algorithm and analyse its runtime and memory complexity.
Lastly, in \cref{sec:review}, Arnoldi aggregations are evaluated and compared to Exlump aggregations from~\cite[Algorithm~3]{michel2025formalbounds} on a range of different models.
A summary and outlook for future work are given in \cref{sec:conclusion}.

%% file: 02-preliminaries.tex
\section{Preliminaries}\label{sec:preliminaries}
\subsection{Notation}\label{subsec:notation}
Let $u,v \in \bbR^n$ and $M \in \bbR^{n \times m}$.
Then, $\abs{v}$ and $\abs{M}$ denote the vector and the matrix with the absolute value applied component-wise.
Furthermore, denote with $\mathbf{0}_n$ and $\mathbf{1}_n$ the vectors in $\bbR^n$ with all entries zero or one, respectively.
Let $e_j \in \bbR^j$ be the $j$th standard basis vector of $\bbR^j$, and
$I_n \in \bbR^{n \times n}$ the identity matrix.
Lastly, define $\inp{u, v} \coloneqq u\transp v$, $\norm{v}_1 \coloneqq \inp{\abs{v}, \mathbf{1}_n}$ and the matrix norm $\norm{M}_\infty$ as the maximum absolute row sum of $M$.

Here, we are going to work with time-homogeneous discrete-time Markov chains with finite state spaces $S = \set{1, \dots, n}$ solely.
Furthermore, if $X_k$ is the state of the Markov chain at time $k$, let $P \in \bbR^{n \times n}$ be the transition matrix of the Markov chain where $P(i, j) \coloneqq \bbP(X_{k + 1} = j \,|\, X_k = i)$.
With an initial distribution $p_0 \in \bbR^n$, transient distributions are given by $p_k\transp = p_0\transp P^k$.

Next, we consider state space reductions through aggregations.
In an aggregation of dimension $m \leq n$, the (arbitrary) \emph{aggregated step matrix} $\Pi \in \bbR^{m \times m}$ shall approximate the dynamics of $P$.
Let $\pi_0 \in \bbR^m$ be the (arbitrary) \emph{initial aggregated vector} and $\pi_k\transp \coloneqq \pi_0\transp\Pi^k$ the \emph{transient aggregated vector}.
To return to an $n$-dimensional space, we use the (arbitrary) \emph{disaggregation matrix} $A \in \bbR^{m \times n}$ to define \emph{approximated transient distributions} as $\tilde{p}_k\transp \coloneqq \pi_k\transp A$.
Finally, following~\cite[Def.~8]{michel2025formalbounds}, an aggregation is \emph{dynamic-exact} if $\Pi A = A P$.
If further $p_0 = \tilde{p}_0$ holds, it is \emph{exact}.
This is motivated by dynamic-exactness implying $\tilde{p}_k\transp = \tilde{p}_0\transp P^k$ and exactness even offering $p_k = \tilde{p}_k$.
The $k$th \emph{transient error vector} is given through $\errvec_k \coloneqq \tilde{p}_k - p_k$ with $\norm{\errvec_k}_1$ denoting the \emph{transient error} after $k$ steps.
\subsection{Arnoldi Iteration}\label{subsec:arnoldi-iteration}
For a vector $v \in \bbR^n$ and a matrix $M \in \bbR^{n \times n}$, the Arnoldi iteration, as first introduced in~\cite{arnoldi1951iteration}, builds an orthonormal basis $q_1\transp, \dots, q_j\transp$ of the Krylov subspace
\[
    \mathcal{K}^j(v, M) \coloneqq \spn\Set{v\transp, v\transp M, \dots, v\transp M^{j - 1}},
\]
using the Gram-Schmidt procedure, as done in~\cref{alg:arnoldi-iteration}.
\begin{algorithm}[ht]
    \caption{Arnoldi Iteration}\label{alg:arnoldi-iteration}
    \begin{algorithmic}[1]
        \State Let $v \in \bbR^n$, $M \in \bbR^{n \times n}$.
        \State $q_1 \coloneqq \frac{v}{\norm{v}_2}$
        \State $Q_1 \coloneqq \begin{pmatrix}\hpipe & q_1\transp & \hpipe\end{pmatrix}$
        \For{$j = 1, 2 \dots$}
            \State $r_1\transp \coloneqq q_j\transp M$
            \ForAll{$i = 1, 2 \dots, j$}
                \State $h_{j, i} \coloneqq \inp{r_i, q_i}$
                \State $r_{i + 1} \coloneqq r_i - h_{j, i} q_i$
            \EndFor
            \State $h_{j, j + 1} \coloneqq \norm{r_{j + 1}}_2$
            \State $H_j \coloneqq
            \begin{cases}
                \begin{pmatrix}h_{1,1}\end{pmatrix} & \text{if $j = 1$}\\
                \begin{pmatrix}
                \multicolumn{2}{c}{\multirow{2}{*}{$H_{j - 1}$}} & \mathbf{0}_{j - 2}\\
                &&h_{j - 1, j}\\
                h_{j, 1} & \cdots & h_{j,j}
            \end{pmatrix} & \text{else}
            \end{cases}$
            \If{criterion is true}
                \Return $H_j$, $Q_j$
            \EndIf
            \State $q_{j + 1} \coloneqq \frac{r_{j + 1}}{h_{j, j + 1}}$
            \State $Q_{j + 1} \coloneqq \begin{pmatrix}&Q_j&\\\hpipe & q_{j + 1}\transp & \hpipe\end{pmatrix}$
        \EndFor
    \end{algorithmic}
\end{algorithm}
There, we can see the Gram-Schmidt procedure in lines~2,~4--9~and~12; lines~3,~10~and~13 construct some matrices $H_j \in \bbR^{j \times j}$ and $Q_j \in \bbR^{j \times n}$, and line~11 provides the return value once some arbitrary termination criterion is met (suggestions for this will follow later).
Usually, $h_{j, j + 1} = 0$ is part of this criterion to avoid division by zero.
The relation
\begin{equation}\label{eq:arnoldi-relation}
H_j Q_j + h_{j, j + 1} e_j q_{j + 1}\transp = Q_j M
\end{equation}
holds if $\mathcal{K}^j(v, M) \subsetneq \mathcal{K}^{j + 1}(v, M)$ (with $e_j \in \bbR^j$).
If $\mathcal{K}^j(v, M) = \mathcal{K}^{j + 1}(v, M)$ holds, (i.e., the Krylov subspace is invariant under $M$), \cref{eq:arnoldi-relation} reduces to
\begin{equation}\label{eq:red-arnoldi-relation}
H_j Q_j = Q_j M.
\end{equation}
While not hard to show, proving~\cref{eq:arnoldi-relation,eq:red-arnoldi-relation} thoroughly requires lengthy calculations.
Refer to~\cite[Sec.~6.2]{saad2011numericalmethods} for a proof.
There, a slight variation of the Arnoldi iteration is used, with rearranged order of operations for the underlying Gram-Schmidt procedure.

%% file: 03-build-aggregations.tex
\section{Using the Arnoldi Iteration to Build Aggregations}\label{sec:building-aggregations}
Both \cref{eq:arnoldi-relation,eq:red-arnoldi-relation}, but especially \cref{eq:red-arnoldi-relation}, are reminiscent of the definition of dynamic-exact aggregations of Markov chains in \cref{subsec:notation}.
Using these as a basis while further guaranteeing $p_0 = \tilde{p}_0$ motivates:
\begin{definition}[Arnoldi aggregation]\label{def:arnoldi-aggregation}
With $P \in \bbR^{n \times n}$ as the transition matrix of a Markov chain and $p_0 \in \bbR^n$ an initial distribution, the aggregation with $\Pi \coloneqq H_j$, $A \coloneqq Q_j$, resulting from the Arnoldi iteration with $P$ and $p_0$, and
$
    \pi_0 \coloneqq (\norm{p_0}_2, 0, \dots, 0)\transp \in \bbR^j
$,
is the \emph{Arnoldi aggregation} of size $j$ for $P$ and $p_0$.
\end{definition}
\subsection{Exact and Initially Exact Arnoldi Aggregations}\label{subsec:exact-and-initially-arnoldi-aggr}
We start by exploring some properties of Arnoldi aggregations and their connection to (initially) exact aggregations.
\begin{lemma}\label{lem:arnoldi-aggregation-pi_k-form}
In an Arnoldi aggregation of size $j$, it holds that
\[
    \pi_k = (\underbrace{\lambda_1, \dots, \lambda_{k + 1},}_{\textrm{all }\in \; \bbR} \underbrace{0, \dots, 0}_{(j - k -1) \times})\transp
    \quad \textrm{for } 1 \leq k \leq j - 1 \textrm{ and for some real numbers } \lambda_i.
\]
\end{lemma}
\begin{proof}
    Follows by induction on $k$ and the form of $H_j$.
\end{proof}
\begin{lemma}\label{lem:closed-form-error}
In Arnoldi aggregations of size $j$ after $k$ steps, we have
\[
    \pi_0\transp H_j^k Q_j + \sum_{i = 0}^{k - 1} \pi_0\transp H_j^i \left(h_{j, j + 1} e_j q_{j + 1}\transp\right) P^{k - 1 - i} = \pi_0\transp Q_j P^k
\]
\end{lemma}
\begin{proof}
    By induction on $k$. Clearly, $\pi_0\transp I_j Q_j + 0 = \pi_0\transp Q_j I_j$. Induction step:
    \begin{align*}
        \pi_0\transp Q_j P^{k + 1} &= \pi_0\transp H_j^k Q_j P + \sum_{i = 0}^{k - 1} \pi_0\transp H_j^i \left(h_{j, j + 1} e_j q_{j + 1}\transp\right) P^{k - i}\\
        &\txteq{\labelcref{eq:arnoldi-relation}} \pi_0\transp H_j^k \left( H_j Q_j + h_{j, j + 1} e_j q_{j + 1}\transp \right) + \sum_{i = 0}^{k - 1} \pi_0\transp H_j^i \left(h_{j, j + 1} e_j q_{j + 1}\transp\right) P^{k - i}\\
        &= \pi_0\transp H_j^{k + 1} Q_j + \sum_{i = 0}^k \pi_0\transp H_j^i \left(h_{j, j + 1} e_j q_{j + 1}\transp\right) P^{k - i}.
    \end{align*}
\end{proof}
\begin{proposition}\label{prop:arnoldi-aggr-inital-exact}
An Arnoldi aggregation of size $j$ has $p_k = \tilde{p}_k$ for $0 \leq k \leq j - 1$.
\end{proposition}
\begin{proof}
    Combine \cref{lem:arnoldi-aggregation-pi_k-form,lem:closed-form-error} to get
    \begin{align*}
        p_k\transp &= \pi_0\transp Q_j P^k \txteq{Lem.~\labelcref{lem:closed-form-error}} \pi_0\transp H_j^k Q_j + \sum_{i = 0}^{k - 1} \pi_0\transp H_j^i \left(h_{j, j + 1} e_j q_{j + 1}\transp\right) P^{k - 1 - i}\\
        &\txteq{Lem.~\labelcref{lem:arnoldi-aggregation-pi_k-form}} \pi_0\transp H_j^k Q_j + \sum_{i = 0}^{k - 1} \mathbf{0}_n\transp P^{k - 1 - i} = \pi_0\transp H_j^k Q_j = \tilde{p}_k\transp.
    \end{align*}
\end{proof}
\cref{prop:arnoldi-aggr-inital-exact} motivates our choice of $\pi_0$ as
it not only guarantees a necessary condition for exactness, $\norm{\errvec_0}_1 = 0$, but extends it to the first $j - 1$ steps.
\begin{definition}[Initial exactness]
    We call an aggregation initially exact, or more precisely, $(j - 1)$-exact, if $\norm{\errvec_k}_1 = 0$ for $0 \leq k \leq j - 1$.
\end{definition}
\begin{remark}
    \cref{lem:closed-form-error,prop:arnoldi-aggr-inital-exact} imply
    \[
        \norm*{\errvec_k}_1 = \norm*{\sum_{i = j - 1}^{k - 1} \pi_0\transp H_j^i \left(h_{j, j + 1} e_j q_{j + 1}\transp\right) P^{k - 1 - i}}_1,
    \]
    providing a closed-form formula for $\norm{\errvec_k}_1$ in Arnoldi aggregations.
\end{remark}
\begin{proposition}\label{prop:arnoldi-aggr-invar-exact}
    An Arnoldi aggregation of size $j$ with the Krylov subspace $\mathcal{K}^j(p_0, P)$ being invariant under $P$ is exact.
\end{proposition}
\begin{proof}
    Dynamic-exactness follows by \cref{eq:red-arnoldi-relation} with invariance of $\mathcal{K}^j(p_0, P)$, and $p_0 = \tilde{p}_0$ follows through \cref{prop:arnoldi-aggr-inital-exact}.
\end{proof}
\begin{theorem}\label{thrm:smallest-initial-exact-aggr}
    An Arnoldi aggregation of size $j$ is a smallest $(j - 1)$-exact aggregation of $P$ and $p_0$.
\end{theorem}
\begin{proof}
    Assume we are given $P \in \bbR^{n \times n}$, $p_0 \in \bbR^n$, and an arbitrary $(j - 1)$-exact aggregation with disaggregation matrix $A$ and aggregated transient vectors $\pi_k$.
    Let $A_{\mathrm{arn}}$ be the disaggregation matrix of the corresponding Arnoldi aggregation of size $j$ with rows $q_i$.
    By initial exactness, we must have for $0 \leq k \leq j - 1$ that $\pi_k\transp A = p_k\transp$.
    Thus,
    \begin{align*}
        & \spn \set{p_0\transp, \dots, p_{j - 1}\transp} \subseteq \rowsp(A)\\
        \implies & \spn \set{q_1\transp, \dots, q_j\transp} = \spn \set{p_0\transp, \dots, p_{j - 1}\transp} \subseteq \rowsp(A).
    \end{align*}
    As by definition $\rowsp(A_\mathrm{arn}) = \spn \set{q_1\transp, \dots, q_j\transp}$, the Arnoldi aggregation of size $j$ is indeed of minimal size because $q_1\transp, \dots, q_j\transp$ are linearly independent.
\end{proof}
\begin{theorem}\label{thrm:smallest-exact-aggr}
    An Arnoldi aggregation of size $j$ with $\mathcal{K}^j(p_0, P)$ being invariant under $P$ is a smallest exact aggregation of $P$ and $p_0$.
\end{theorem}
\begin{proof}
    Follows analogous to \cref{thrm:smallest-initial-exact-aggr}.
    Given an exact aggregation of arbitrary size $m$ with disaggregation matrix $A$, we know that $\spn \set{p_0\transp, p_1\transp, \dots} \subseteq \rowsp(A)$.
    As $q_1\transp, \dots q_j\transp$ is a basis of $\spn \set{p_0\transp, p_1\transp, \dots}$, the Arnoldi aggregation of size $j$ with $q_1\transp, \dots q_j\transp$ as rows of the disaggregation matrix is minimal, i.e., $j \leq m$.
\end{proof}

\subsection{Determining Convergence}\label{subsec:convergence}
Recall that in \cref{alg:arnoldi-iteration}, we have left the criterion for termination open.
Naturally, with the definition of exactness requiring $\norm{H_j Q_j - Q_j P}_\infty = 0$ and \cref{eq:arnoldi-relation,eq:red-arnoldi-relation}, $\norm{H_j Q_j - Q_j P}_\infty = 0$, $\abs{h_{j, j + 1}} = 0$, or $\norm{r_{j+1}}_1 = 0$ offer themselves as intuitive measurements for detecting exactness.

Unfortunately, none of these approaches work in general, which we illustrate by looking at an example \abbrev{dtmc}.
The \abbrev{rsvp} model is a stochastic process algebra model from~\cite{wang2008rsvp}.
It is composed of a lower channel submodel with capacity $M$, an upper channel submodel with capacity $N$, and some instances of mobile nodes making call requests at a constant rate.
For $M = 7$, $N = 5$ and three mobile nodes, resulting in 842 states, symmetries among these nodes enable an exact aggregation with $j = 234$ for the \abbrev{dtmc} resulting from uniformisation of the original \abbrev{ctmc} with uniformisation rate 30.01~\cite[Sec.~6.4]{michel2025formalbounds}.
Then, by~\cref{thrm:smallest-exact-aggr}, \cref{alg:arnoldi-iteration} must find an exact Arnoldi aggregation at size $j = 234$ or lower.
\begin{figure}[ht]
    \centering
    \begin{tikzpicture}
        \begin{axis}[
            xlabel = {Dimension of Arnoldi aggregation},
            ylabel = {},
            xmin = 1,
            xmax = 251,
            ymin = 0,
            ymax = 5,
            legend pos = south west,
            width = 0.95\linewidth,
            height = 6.0cm,
            colorblindstyle,
        ]
            \addplot coordinates {
                (1,0.0) (11,1.0908799048408153) (21,1.2714725299086103) (31,2.3994343720804068) (41,3.3744424549067644) (51,1.974564734099342) (61,1.4549958539328893) (71,2.635886423931421) (81,3.327572545905713) (91,3.826194437012354) (101,2.359661634688644) (111,1.316392540909928) (121,2.326619350850972) (131,4.791573165296705) (141,1.8134754472282537) (151,1.44818085656352) (161,3.2090732963025825) (171,3.1625529157131225) (181,2.3861462894055463) (191,2.1848736951302916) (201,2.656985284177068) (211,4.124739914381117) (221,2.6615352362188323) (231,2.6580439336644197) (241,1.5070084238839205) (251,3.15919115005751)
            };
            \addlegendentry{$\norm{r_{j+1}}_1$}
            \addplot coordinates {
                (1,0.0) (11,1.7010011406551544) (21,2.025096706540314) (31,3.0686704741675) (41,3.833888678739754) (51,2.3271081747220235) (61,1.6055664357611201) (71,2.932480575916255) (81,3.544326773432891) (91,4.32162697804645) (101,2.4655644383511186) (111,1.6291043643678447) (121,2.496000679955839) (131,4.918729884230226) (141,2.2529847240200693) (151,1.4559098663553325) (161,4.434868689874161) (171,3.205434096032204) (181,2.9354906853395413) (191,2.415205983321618) (201,3.147515801836711) (211,3.3011791421735546) (221,2.7553590032125497) (231,2.6826180515555307) (241,1.4444628946304554) (251,2.5820964319521944) 
            };
            \addlegendentry{$15 \cdot \abs{h_{j, j + 1}}$}
            \addplot coordinates {
                (1,0.000999666777740753) (11,2.6186341020987425) (21,2.823274212986767) (31,2.5502164671553835) (41,3.9067566945008925) (51,2.7086147742752216) (61,2.0739989045969662) (71,3.3026168997547765) (81,4.007103302643773) (91,1.8931423632770408) (101,2.9189107220978103) (111,1.5614694534830804) (121,2.443918314513966) (131,4.347014215643402) (141,2.015266885626381) (151,1.760304728549291) (161,3.324428662917293) (171,3.1680021890030363) (181,2.1281758354743068) (191,2.1792565825551895) (201,2.8346269837043216) (211,4.497026283406091) (221,2.007513863042594) (231,2.1897137922814416) (241,2.1095958860124826) (251,3.540177864152146)
            };
            \addlegendentry{$\norm{H_j Q_j - Q_j P}_\infty$}
        \end{axis}
    \end{tikzpicture}
    \caption{Different values related to $\norm{H_j Q_j - Q_j P}_\infty$ depending on the dimension of Arnoldi aggregations of the \abbrev{rsvp} model.}
    \label{fig:rsvp-no-conv}
\end{figure}
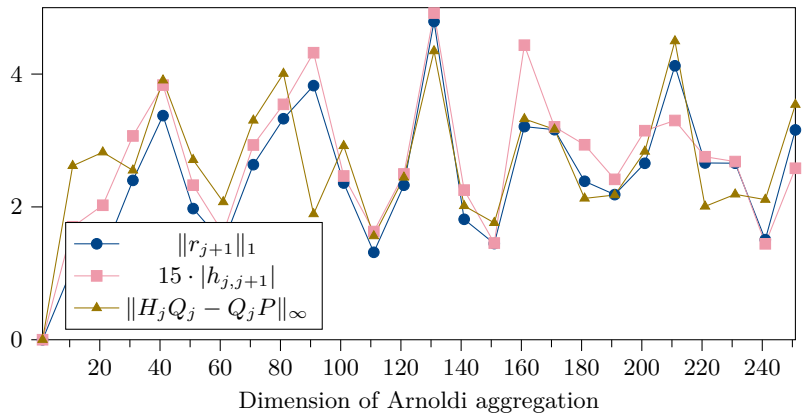
\Cref{fig:rsvp-no-conv} shows the proposed naive termination criteria for $p_0 = (1, 0, \dots, 0)\transp \in \bbR^{842}$, for which an exact aggregation as described exists.
Note that $\abs{h_{j, j + 1}}$ is multiplied by 15 in \Cref{fig:rsvp-no-conv} simply to match the size of the other two criteria.
Since there is obviously no striking behaviour near $j = 234$, we must take a different approach altogether.

By~\cite[Thm.~4~(i)]{michel2025formalbounds} and \cref{prop:arnoldi-aggr-inital-exact}, the error bound
\begin{equation}\label{eq:arnoldi-bound}
    \norm{\errvec_k}_1 \leq \sum_{j = 0}^{k - 1} \inp*{\abs{\pi_k}, \abs{H_j Q_j - Q_j P} \cdot \mathbf{1}_n}
\end{equation}
holds for all Arnoldi aggregations.
Note that, in general, \cref{eq:arnoldi-bound} cannot be improved.
Although \cref{eq:arnoldi-bound} offers the tightest general bound, its computational cost is too high for a usable criterion to determine whether to stop the expansion of an Arnoldi aggregation.

For an exact aggregation, it must hold that $\inp*{\abs{\pi_k}, \abs{H_j Q_j - Q_j P} \cdot \mathbf{1}_n} = 0$ for all $k$.
Under the assumption that $(\pi_k)_{k \in \bbN}$ converges to an eigenvector $\pi$, it must then hold that $\inp*{\abs{\pi}, \abs{H_j Q_j - Q_j P} \cdot \mathbf{1}_n} = 0$.
Importantly, this is not an equivalence and the assumption about the convergence of $(\pi_k)_{k \in \bbN}$ does not have to hold either.
Still, we propose to use $\inp*{\abs{\pi}, \abs{H_j Q_j - Q_j P} \cdot \mathbf{1}_n} \leq \varepsilon$ for some given $\varepsilon \in \bbR_0^+$ as a stopping criterion.

Per~\cite[Sec.~6.7]{saad2011numericalmethods}, the largest eigenvalues of $H_j$ eventually yield very good approximations of the largest eigenvalues of $P$ for large enough $j$.
In this case, both $H_j$ and $P$, as a transition matrix, will have largest eigenvalue one.
Thus, we use the Krylov-Schur method presented in~\cite{stewart2002krylovschur} to determine $\pi$ with eigenvalue one.
If there are multiple such $\pi$, we choose the one whose eigenvalue is numerically closest to one.
If $\pi$ has complex entries, hinting at improper convergence, we immediately further expand the aggregation.
In practice, this happens only at $j$ so small that the resulting aggregation is unusable.

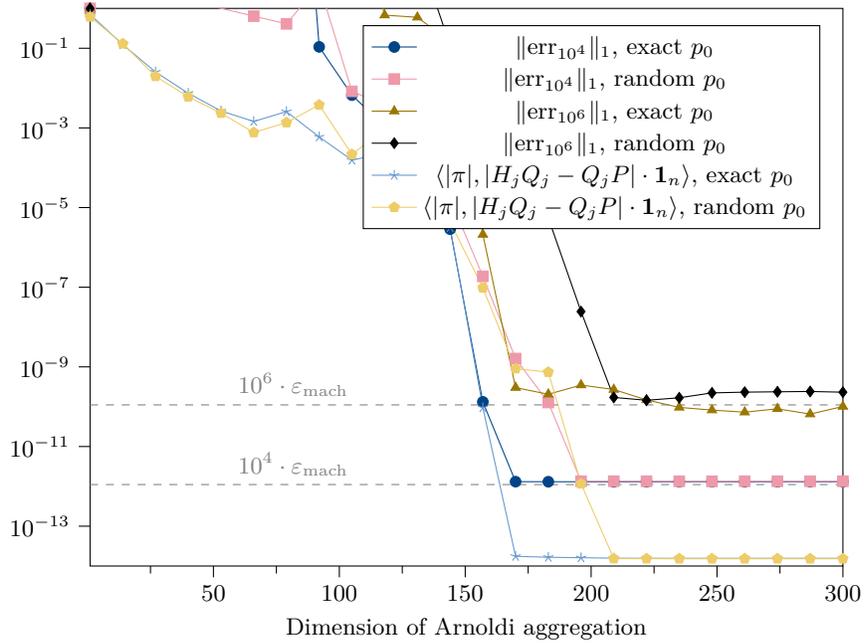
\begin{figure}[ht]
    \centering
    \begin{tikzpicture}
        \begin{axis}[
            xlabel = {Dimension of Arnoldi aggregation},
            ymode = log,
            xmin = 1,
            xmax = 300,
            ymin = 1e-14,
            ymax = 1,
            domain=1:300,
            legend pos = north east,
            width = 0.95\linewidth,
            height = 9.0cm,
            colorblindstyle,
        ]
            \addplot coordinates {
                (1,0.9999999999996783) (14,1.0590014209731492e93) (27,1.093927771271882e9) (40,12.935893732569744) (53,1.601882559284899) (66,1.0263999790864988) (79,4.440360703614884e9) (92,0.10823023233173526) (105,0.006656381653855471) (118,0.0011022447661217093) (131,0.00012873844936988325) (144,2.868205016207439e-6) (157,1.3213192209434936e-10) (170,1.3114389255293993e-12) (183,1.3070378506236906e-12) (196,1.303811571017354e-12) (209,1.305917127743144e-12) (222,1.3076426473347334e-12) (235,1.3091230464697159e-12) (248,1.30885011763483e-12) (261,1.308850117634904e-12) (274,1.308850117634904e-12) (287,1.308850117634904e-12) (300,1.308850117634904e-12)
            };
            \addlegendentry{$\norm{\errvec_{10^4}}_1$, exact $p_0$}
            \addplot coordinates {
                (1,0.9999999999996796) (14,8.682901937131364e17) (27,2465.614499136977) (40,1.499933270633242) (53,1.0019176676563952) (66,0.641075662678966) (79,0.4098992883744737) (92,3.9380877553089775) (105,0.008280336890477236) (118,0.0037067048536584325) (131,0.0004946966456049212) (144,1.6735135945241204e-5) (157,1.8784141269801397e-7) (170,1.6130074700486288e-9) (183,1.276130098916059e-10) (196,1.3477150728000954e-12) (209,1.3494092989391435e-12) (222,1.3454557357461826e-12) (235,1.3445997750046378e-12) (248,1.3382083761360866e-12) (261,1.347495385878242e-12) (274,1.34417131793331e-12) (287,1.34417131793331e-12) (300,1.34417131793331e-12)
            };
            \addlegendentry{$\norm{\errvec_{10^4}}_1$, random $p_0$}
            \addplot coordinates {
                (1,0.9999999999971305) (14,NaN) (27,NaN) (40,2.6210701366959337e127) (53,7.382732221673995e65) (66,3.47238182454086e242) (79,4.1361829948267124e52) (92,NaN) (105,1049.964897942284) (118,0.6778038001260972) (131,0.5982002267321362) (144,0.14791531930119853) (157,2.1103031129639635e-6) (170,3.0179119464193626e-10) (183,2.0514197359653264e-10) (196,3.499224377307924e-10) (209,2.6997817469993446e-10) (222,1.4807839176906262e-10) (235,9.612420811688927e-11) (248,8.233896140165956e-11) (261,7.276852916613362e-11) (274,8.873557096644478e-11) (287,6.527061438350924e-11) (300,1.0133929029448865e-10)
            };
            \addlegendentry{$\norm{\errvec_{10^6}}_1$, exact $p_0$}
            \addplot coordinates {
                (1,0.9999999999971305) (14,NaN) (27,NaN) (40,8.982260194890708e107) (53,1.504391780018984e73) (66,2.773861182232495e21) (79,9.616735326195112e14) (92,NaN) (105,154368.8085179561) (118,NaN) (131,31.16255372603714) (144,0.1160158113476204) (157,0.00195668393705069) (170,1.734235506987551e-5) (183,5.459926185828008e-6) (196,2.4494770614446536e-8) (209,1.709799700882245e-10) (222,1.4506467556195953e-10) (235,1.6746438345988565e-10) (248,2.2211330033513035e-10) (261,2.321736448038248e-10) (274,2.368279248713506e-10) (287,2.4114263538354774e-10) (300,2.3142362550828243e-10)
            };
            \addlegendentry{$\norm{\errvec_{10^6}}_1$, random $p_0$}
            \addplot coordinates {
                (1,0.7041328695662974) (14,0.12342073968442469) (27,0.02532684971642423) (40,0.007408188199569539) (53,0.002666716652444382) (66,0.0014533531651663705) (79,0.002564964200176851) (92,0.0005946063531233087) (105,0.00015526454138921808) (118,0.00023040410830026) (131,5.091677994027011e-5) (144,3.164163069191886e-6) (157,9.429737820802962e-11) (170,1.7513942952862348e-14) (183,1.6594556887428333e-14) (196,1.6161132854087552e-14) (209,1.57220080134689e-14) (222,1.562629761697312e-14) (235,1.5613699267473614e-14) (248,1.5612424595228052e-14) (261,1.5612177757245902e-14) (274,1.5612177760215623e-14) (287,1.5612177758025978e-14) (300,1.5612177758066687e-14)
            };
            \addlegendentry{$\inp*{\abs{\pi}, \abs{H_j Q_j - Q_j P} \cdot \mathbf{1}_n}$, exact $p_0$}
            \addplot coordinates {
                (1,0.61271435355893) (14,0.13008844719682303) (27,0.019734373894162917) (40,0.006051677593685434) (53,0.002347030137527128) (66,0.0007684477820357159) (79,0.001361972498503312) (92,0.003775724460595941) (105,0.00021692615194793293) (118,0.0009337011778195679) (131,0.00018333597018766514) (144,4.3384036873767105e-6) (157,9.65952509301048e-8) (170,9.021946924875666e-10) (183,7.38542836959882e-10) (196,1.161515835346837e-12) (209,1.5337653260138767e-14) (222,1.5254605832760032e-14) (235,1.5238685743953203e-14) (248,1.522610691894044e-14) (261,1.522603265192457e-14) (274,1.52259987841452e-14) (287,1.5225998772688912e-14) (300,1.5225998773445174e-14)
            };
            \addlegendentry{$\inp*{\abs{\pi}, \abs{H_j Q_j - Q_j P} \cdot \mathbf{1}_n}$, random $p_0$}
            \addplot+[no marks, gray, dashed]
                {1.11 * 1e-12}
                node[pos=0.27, above] {$10^4 \cdot \varepsilon_\mathrm{mach}$};
            \addplot+[no marks, gray,dashed]
                {1.11 * 1e-10}
                node[pos=0.27, above] {$10^6 \cdot \varepsilon_\mathrm{mach}$};
        \end{axis}
    \end{tikzpicture}
    \caption{Proposed convergence criterion and error after $10^4$ and $10^6$ steps in the \abbrev{rsvp} model, depending on the dimension of the Arnoldi aggregation.}
    \label{fig:rsvp-conv-crit}
\end{figure}

This results in \cref{fig:rsvp-conv-crit}, where we either used random $p_0$ or initial distributions for which we know that there is an exact aggregation at $j = 234$.
Furthermore, we used ten samples per data point and $\varepsilon_\mathrm{mach}$ is the \abbrev{ieee~754} double precision rounding machine epsilon.
As such, there is a visible correlation between $\norm{\errvec_k}_1$ and $\inp*{\abs{\pi}, \abs{H_j Q_j - Q_j P} \cdot \mathbf{1}_n}$, although no formal proof or connection between these two is known.
Still, over a range of models (see \cref{sec:review}), we were not able to find a case where the general trend, as seen, does not hold.

The algorithm to compute Arnoldi aggregations with $p_k \approx \tilde{p}_k$ is the same as \cref{alg:arnoldi-iteration} but with line 11 replaced, as seen in \cref{alg:arnoldi-aggr-crit}.
\begin{algorithm}[ht]
    \caption{Finding an Arnoldi aggregation with $\inp*{\abs{\pi}, \abs{H_j Q_j - Q_j P} \cdot \mathbf{1}_n} \leq \varepsilon$}\label{alg:arnoldi-aggr-crit}
    \begin{algorithmic}[1]
        \State Let $\varepsilon \in \bbR_0^+$, $p_0 \in \bbR^n$, $P \in \bbR^{n \times n}$.
        \State $q_1 \coloneqq \frac{p_0}{\norm{p_0}_2}$
        \State $Q_1 \coloneqq \begin{pmatrix}\hpipe & q_1\transp & \hpipe\end{pmatrix}$
        \For{$j = 1, 2 \dots$}
            \State $r_1\transp \coloneqq q_j\transp P$
            \State \dots \Comment{\cref{alg:arnoldi-iteration}, lines 6--10}
            \If{$j \pmod{10} = 0$}
                \State Compute dominant eigenvector $\pi\transp$ of $H_j$ as in~\cite{stewart2002krylovschur}
                \State $\pi = \frac{\pi}{\norm*{\pi\transp Q_j}_1}$
                \If{$\pi \in \bbR^j$ and $\inp*{\abs{\pi}, \abs{H_j Q_j - Q_j P} \cdot \mathbf{1}_n} \leq \varepsilon$}
                    \State $\pi_0 \coloneqq (\norm{p_0}_2, 0, \dots, 0)\transp \in \bbR^j$
                    \State \Return $H_j$, $Q_j$, $\pi_0$, $\pi$
                \EndIf
            \EndIf
            \State \dots \Comment{\cref{alg:arnoldi-iteration}, lines 12 and 13}
        \EndFor
    \end{algorithmic}
\end{algorithm}
The convergence criterion is naively applied only every tenth expansion to reduce the runtime.

We can also easily analyse the runtime.
If $P$ is dense, the vector-matrix multiplication in line 5 is in $\mathcal{O}(n^2)$, if $P$ is sparse, it is in $\mathcal{O}(n)$ instead.
The outer for loop is, by definition, repeated $j$ times and the inner one on average $\frac{j}{2}$ times.
Furthermore, the inner product in the inner for loop takes $\mathcal{O}(n)$.
Thus, without the convergence criterion, we are in $\mathcal{O}(n^2j + nj^2)$ for dense $P$ or $\mathcal{O}(nj^2)$ for sparse $P$.
In the computation needed for the convergence criterion, the calculation of $\pi$ is the dominant factor.
There, the Arnoldi iteration is applied to $H_j$ a finite number of times (independent of $H_j$ and $j$) with a random starting vector in our implementation.
As $H_j$ is dense, this is in $\mathcal{O}(j^2)$.
Although further steps are taken to compute $\pi$, they all end up in constant time, even though with a fairly large constant.
Overall, we are in $\mathcal{O}(n^2 j + n j ^2 + j^3)$ for dense $P$ and $\mathcal{O}(n j ^2 + j^3)$ for sparse $P$.
The memory complexity is the same, asymptotically.

%% file: 04-review.tex
\section{Reviewing Arnoldi Aggregations}\label{sec:review}
Now that we have devised \cref{alg:arnoldi-aggr-crit}, evaluating its performance is of great interest.
An implementation in Julia (see~\cite{bezanson2017julia}) version 1.11.2, using the package KrylovKit~\cite{haegeman2024krylovkit} for the underlying Arnoldi iteration and for the Krylov-Schur method from~\cite{stewart2002krylovschur}, can be found at~\cite{sonnentag2025implementation}.
Using this implementation, we compare \cref{alg:arnoldi-aggr-crit} to the Exlump algorithm~\cite[Algorithm~3]{michel2025formalbounds}.
This algorithm works by forming an explicit partition of the original state space, such that for any two states within the same aggregate, their incoming probabilities are close in some way.
We use a so-called uniform disaggregation matrix $A$ in Exlump aggregations, where $A(i, j)$ is the reciprocal of the size of the $i$th aggregate~\cite[p.~14]{michel2025formalbounds}.

Apart from the already introduced \abbrev{rsvp} model, we compare these two algorithms on a range of Markov chains.
For one, we use a Lotka-Volterra model (see~\cite{gillespie1977lotkavolterra} for example) with a maximum population of 100, resulting in 10{,}201 states, which is uniformised at a rate of 2{,}078.
We also look at a Markov chain with 15{,}540 states, modelling two workstation clusters with 20 workstations per cluster, where each workstation can fail and be repaired~\cite{haverkort2000workstation}.
There, the uniformisation rate is 50.08.
Finally, our largest Markov chain is given by a prokaryotic gene expression model~\cite{kierzek2001genemodel} with a maximal population size of five and a state space of 43{,}957, uniformised at a rate of 16.78.
Note that the Exlump algorithm was implemented in Python, whereas we compute Arnoldi aggregations by using Julia.
Thus, runtimes cannot be compared directly between these two.

We compare transient errors of the \abbrev{rsvp} model in \cref{fig:errorrsvpcomp}, where we see that Exlump aggregations perform similarly to Arnoldi aggregations, with the first usable aggregations at $j \approx 130$.
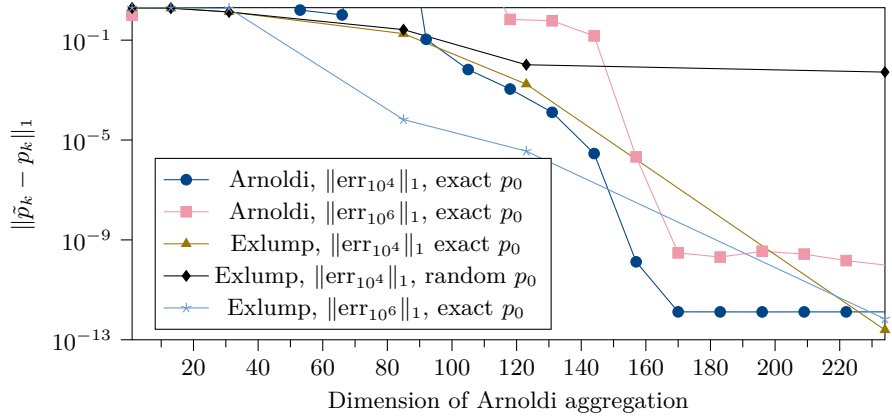
\begin{figure}[ht]
    \centering
    \begin{tikzpicture}
        \begin{axis}[
            xlabel = {Dimension of Arnoldi aggregation},
            ylabel = {$\norm{\tilde{p}_k - p_k}_1$},
            ymode = log,
            xmin = 1,
            xmax = 234,
            ymin = 1e-13,
            ymax = 2,
            legend pos = south west,
            width = 0.95\linewidth,
            height = 6.0cm,
            colorblindstyle,
        ]
            \addplot coordinates {
                (1,0.9999999999996783) (14,1.0590014209731492e93) (27,1.093927771271882e9) (40,12.935893732569744) (53,1.601882559284899) (66,1.0263999790864988) (79,4.440360703614884e9) (92,0.10823023233173526) (105,0.006656381653855471) (118,0.0011022447661217093) (131,0.00012873844936988325) (144,2.868205016207439e-6) (157,1.3213192209434936e-10) (170,1.3114389255293993e-12) (183,1.3070378506236906e-12) (196,1.303811571017354e-12) (209,1.305917127743144e-12) (222,1.3076426473347334e-12) (235,1.3091230464697159e-12)
            };
            \addlegendentry{Arnoldi, $\norm{\errvec_{10^4}}_1$, exact $p_0$}
            \addplot coordinates {
                (1,0.9999999999971305) (14,NaN) (27,NaN) (40,2.6210701366959337e127) (53,7.382732221673995e65) (66,3.47238182454086e242) (79,4.1361829948267124e52) (92,NaN) (105,1049.964897942284) (118,0.6778038001260972) (131,0.5982002267321362) (144,0.14791531930119853) (157,2.1103031129639635e-6) (170,3.0179119464193626e-10) (183,2.0514197359653264e-10) (196,3.499224377307924e-10) (209,2.6997817469993446e-10) (222,1.4807839176906262e-10) (235,9.507454990465803e-11)
            };
            \addlegendentry{Arnoldi, $\norm{\errvec_{10^6}}_1$, exact $p_0$}
            \addplot coordinates {
                (1,1.949928382850986) (13,1.9050753657561208) (31,1.3513770222377024) (85,0.18207581740237608) (123,0.001733212465603741) (234,2.483729145080202e-13)
            };
            \addlegendentry{Exlump, $\norm{\errvec_{10^4}}_1$ exact $p_0$}
            \addplot coordinates {
                (1,1.9498808980075593) (13,1.9098973279096165) (31,1.3409567392481854) (85,0.26378462546805925) (123,0.010335033836082565) (234,0.005262650500545417)
            };
            \addlegendentry{Exlump, $\norm{\errvec_{10^4}}_1$, random $p_0$}
            \addplot coordinates {
                (1,1.9776035540749521) (13,1.9724002080337637) (31,1.8633400696449895) (85,6.583851230470933e-5) (123,3.5792770755141195e-6) (234,6.804311490889649e-13)
            };
            \addlegendentry{Exlump, $\norm{\errvec_{10^6}}_1$, exact $p_0$}
        \end{axis}
    \end{tikzpicture}
    \caption{$\norm{\errvec_{10^4}}_1$ and $\norm{\errvec_{10^6}}_1$ in the \abbrev{rsvp} model for Arnoldi and Exlump aggregations, with different initial distributions, depending on the aggregation dimension.}
    \label{fig:errorrsvpcomp}
\end{figure}
The only exception occurs with random $p_0$ after $10^4$ steps in Exlump aggregations.
This is due to Exlump aggregations being formed independently of $p_0$, but with only specific $p_0$ enabling exact aggregations.

\afterpage{\clearpage}

In contrast to the similar error performance, the runtime analysis in \cref{fig:runtimersvpcomp} shows significant differences.
Note that the different line heights indicating the time needed to compute $p_{10^5}$ without aggregation are caused by Exlump and \cref{alg:arnoldi-iteration,alg:arnoldi-aggr-crit} being implemented in different programming languages.
In \cref{fig:runtimersvpcomp} and the following figures, \cref{alg:arnoldi-iteration} is run for some pre-specified number of iterations without any termination criterion.
Exlump allows to compute $\tilde{p}_{10^5}$ consistently faster than computing $p_{10^5}$ naively.
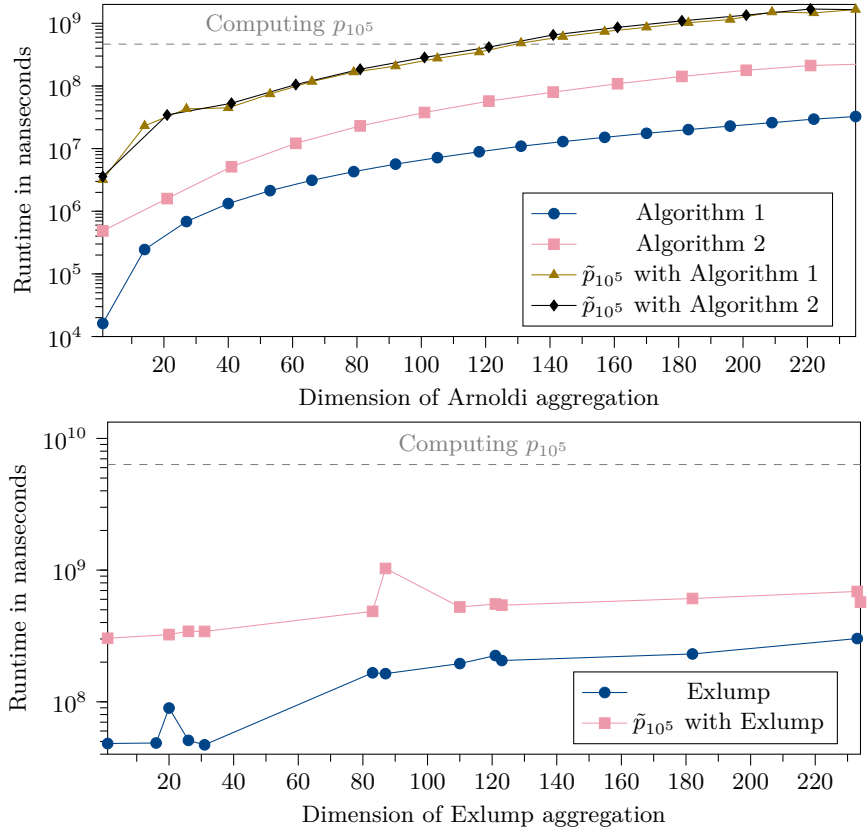
\begin{figure}[ht]
    \centering
    \begin{tikzpicture}
        \begin{axis}[
            xlabel = {Dimension of Arnoldi aggregation},
            ylabel = {Runtime in nanseconds},
            ymode = log,
            xmin = 1,
            xmax = 235,
            ymin = 1e4,
            ymax = 2e9,
            domain=1:235,
            legend pos = south east,
            width = 0.95\linewidth,
            height = 6.0cm,
            colorblindstyle,
        ]
            \addplot coordinates {
                (1,16167.0) (14,244874.99999999997) (27,683708.0) (40,1.327875e6) (53,2.132958e6) (66,3.118167e6) (79,4.288208e6) (92,5.648042e6) (105,7.1655e6) (118,8.897708e6) (131,1.0921125e7) (144,1.2932e7) (157,1.5126333e7) (170,1.7543666e7) (183,2.0087167e7) (196,2.2875292e7) (209,2.5942334e7) (222,2.9546458e7) (235,3.2515125e7)
            };
            \addlegendentry{\cref{alg:arnoldi-iteration}}
            \addplot coordinates {
                (1,485375) (21,1592000) (41,5131000) (61,12145000) (81,23031000) (101,37723000) (121,57244000) (141,79721000) (161,108438000) (181,142017000) (201,177866000) (221,212077000) (241,226023000)
            };
            \addlegendentry{\cref{alg:arnoldi-aggr-crit}}
            \addplot coordinates {
                (1,3.186125e6) (14,2.3111291e7) (27,4.3066084e7) (40,4.466575e7) (53,7.4442042e7) (66,1.17590542e8) (79,1.66460541e8) (92,2.07320792e8) (105,2.76531333e8) (118,3.46874208e8) (131,4.83037583e8) (144,6.13090917e8) (157,7.38502584e8) (170,8.64359417e8) (183,1.022389583e9) (196,1.1495105e9) (209,1.511606833e9) (222,1.464653667e9) (235,1.658823291e9)
            };
            \addlegendentry{$\tilde{p}_{10^5}$ with \cref{alg:arnoldi-iteration}}
            \addplot coordinates {
                (1,3584000) (21,34431000) (41,52960000) (61,105238000) (81,184840000) (101,285200000) (121,416370000) (141,652155000) (161,861237000) (181,1094000000) (201,1351000000) (221,1690000000) (241,1632000000)
            };
            \addlegendentry{$\tilde{p}_{10^5}$ with \cref{alg:arnoldi-aggr-crit}}
            \addplot+[no marks, gray, dashed]{4.64735667e8}
                node[pos=0.25, above] {Computing $p_{10^5}$};
        \end{axis}
    \end{tikzpicture}
    \begin{tikzpicture}
        \begin{axis}[
            xlabel = {Dimension of Exlump aggregation},
            ylabel = {Runtime in nanseconds},
            ymode = log,
            xmin = 1,
            xmax = 234,
            ymin = 4e7,
            ymax = 1.33e10,
            domain=1:234,
            legend pos = south east,
            width = 0.95\linewidth,
            height = 6.0cm,
            colorblindstyle,
        ]
            \addplot coordinates {
                (234 170905113.22021484) (233,302190303.80249023) (182,230839967.72766113) (123,205775260.92529297) (121,224488973.6175537) (110,195050954.8187256) (87,163579940.79589844) (83,165996074.67651367) (31,47124862.67089844) (26,50933122.634887695) (20,89674949.6459961) (16,48655033.111572266) (1,48215779.35736731)
            };
            \addlegendentry{Exlump}
            \addplot coordinates {
                (234,570801973.3428955) (233,688514709.4726562) (182,608697891.2353516) (123,541729927.0629883) (121,552939891.8151855) (110,524688005.4473877) (87,1029531002.0446777) (83,486077308.65478516) (31,342112064.36157227) (26,342994213.10424805) (20,323736906.05163574) (1,303632351.47033436)
            };
            \addlegendentry{$\tilde{p}_{10^5}$ with Exlump}
            \addplot+[no marks, gray, dashed]{6337805747}
                node[pos=0.5, above] {Computing $p_{10^5}$};
        \end{axis}
    \end{tikzpicture}
    \caption{Runtime of computing Arnoldi and Exlump aggregations of the \abbrev{rsvp} model, both with and without computing $\tilde{p}_{10^5}$, depending on the dimension of the aggregations.}
    \label{fig:runtimersvpcomp}
\end{figure}
While the Exlump algorithm itself is, relatively speaking, faster and scales better for larger $j$ than \cref{alg:arnoldi-iteration,alg:arnoldi-aggr-crit}, it also offers a more important advantage: the $\Pi$ of Exlump aggregations retains the sparsity of $P$, whereas the $H_j$ of an Arnoldi aggregation is dense by definition.
Thus, computing approximated transient distributions is much faster for Exlump aggregations. In comparison, it accounts for around $85\%$ of the total runtime needed to compute $\tilde{p}_{10^5}$ in Arnoldi aggregations.
The criterion from \cref{alg:arnoldi-aggr-crit} takes around $12\%$ of the total runtime with the remaining $3\%$ needed for the actual underlying Arnoldi iteration.

We continue with the errors in the workstation cluster model in \cref{fig:errorclustercomp}, where the behaviour is different.
\begin{figure}[ht]
    \centering
    \begin{tikzpicture}
        \begin{axis}[
            xlabel = {Dimension of Arnoldi aggregation},
            ylabel = {$\norm{\tilde{p}_k - p_k}_1$},
            ymode = log,
            xmin = 3,
            xmax = 5523,
            ymin = 1e-15,
            ymax = 2,
            legend pos = south east,
            width = 0.95\linewidth,
            height = 6.0cm,
            colorblindstyle,
        ]
            \addplot coordinates {
                (1,0.9999999999997736) (101,2.9960676944945974) (201,1.0974347226627747e-5) (301,1.1665471161805482e-9) (401,4.7764182672623035e-12) (801,4.655221907811145e-12) (3000,4.650178456283257e-12) (5523,4.65502374565145e-12)
            };
            \addlegendentry{Arnoldi, $\norm{\errvec_{10^4}}_1$, exact $p_0$}
            \addplot coordinates {
                (1,0.9999999999997214) (101,0.9999999999997214) (201,0.003193343448207156) (301,1.0636349656083822e-7) (401,5.426755841233736e-10) (801,5.474716192697255e-10) (3000,5.580831224640294e-10) (5523,5.601280670949976e-10)
            };
            \addlegendentry{Arnoldi, $\norm{\errvec_{10^6}}_1$, exact $p_0$}
            \addplot coordinates {
                (3,1.996267487746774) (93,0.08103739715454) (447,0.002751336696423843) (813,0.002283314385808007) (1021,0.0022998681863000453) (1135,0.011948128429298623) (1567,0.011638325911909523) (1803,0.00228743294468368) (2117,0.011701582406766284) (2584,0.0010107578319757582) (2658,0.00023701249356537527) (2778,4.897748750512534e-5) (3010,0.0002005197398484969) (3073,0.0012253542411099266) (4099,4.723290446820298e-5) (5523,7.276165492753359e-15)
            };
            \addlegendentry{Exlump, $\norm{\errvec_{10^4}}_1$, exact $p_0$}
            \addplot coordinates {
                (3,1.9962674417491288) (93,0.08103626034232529) (447,0.002751336684639088) (813,0.0022833143002223754) (1021,0.0022998675979318825) (1135,0.011948128555891401) (1567,0.011638325850456214) (1803,0.0022874326423915296) (2117,0.011701582159996388) (2584,0.0010107575552173362) (2658,0.0002370122029990647) (2778,4.897761123273923e-5) (3010,0.00020051943143081547) (3073,0.001225354205629224) (4099,4.723293679567371e-5) (5523,1.4543713377705557e-13)
            };
            \addlegendentry{Exlump, $\norm{\errvec_{10^6}}_1$, exact $p_0$}
        \end{axis}
    \end{tikzpicture}
    \caption{$\norm{\errvec_{10^4}}_1$ and $\norm{\errvec_{10^6}}_1$ in the workstation cluster model for Arnoldi and Exlump aggregations, depending on the aggregation dimension.}
    \label{fig:errorclustercomp}
\end{figure}
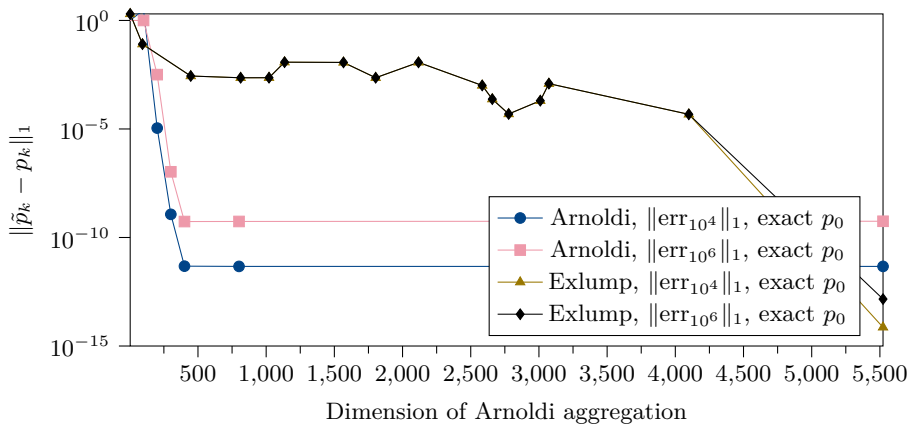
Arnoldi aggregations have already reached their optimum at $j \approx 400$, beating Exlump aggregations by far until their exact aggregation at $j = 5{,}523$.
Notably, in Arnoldi aggregations we consistently have $k  \varepsilon_{\mathrm{mach}} \approx \norm{\errvec_k}_1$ for $j \gtrsim 400$, while Exlump aggregations can eventually achieve a lower error.
This can be explained by $\tilde{p}_k$ converging toward the actual stationary distribution in Exlump aggregations (because the actual stationary distribution is compatible with the exact aggregation at $j=5{,}523$, see~\cite[Theorem~3]{buchholz1994lumpability}), which is not guaranteed for Arnoldi aggregations of the size considered here.
Generally speaking, random $p_0$ yield the same effect as shown here for those $p_0$ which enable an exact aggregation at $j = 5{,}523$.

While computing $\tilde{p}_k$ still takes the most time for the workstation cluster model in \cref{fig:runtimeclustercomp} with the Arnoldi aggregation, its share went down to about 45\% with the convergence criterion now needing roughly 40\% of the total runtime.
Still, the results are much better than in the \abbrev{rsvp} model.
\begin{figure}[ht]
    \centering
    \begin{tikzpicture}
        \begin{axis}[
            xlabel = {Dimension of Arnoldi aggregation},
            ylabel = {Runtime in nanseconds},
            ymode = log,
            xmin = 1,
            xmax = 505,
            ymin = 1e5,
            ymax = 4e10,
            domain=1:505,
            legend pos = south east,
            width = 0.95\linewidth,
            height = 6.0cm,
            colorblindstyle,
        ]
            \addplot coordinates {
                (1,283167.0) (43,1.8899417e7) (85,6.3126959e7) (127,1.15376875e8) (169,2.13173625e8) (211,3.74552667e8) (253,5.0098699999999994e8) (295,7.24549041e8) (337,9.57630625e8) (379,1.236357083e9) (421,1.528603083e9)(463,1.911957709e9) (505,2.470805834e9)
            };
            \addlegendentry{\cref{alg:arnoldi-iteration}}
            \addplot coordinates {
                (1,27017000) (51,93335000) (101,244882000) (151,611977000) (201,1190000000) (251,1866000000) (301,2845000000) (351,3983000000) (401,5521000000) (451,7035000000) (501,9208000000) (551,11552000000)
            };
            \addlegendentry{\cref{alg:arnoldi-aggr-crit}}
            \addplot coordinates {
                (1,3.928208e6) (43,8.4052625e7) (85,2.42091541e8) (127,5.1614574999999994e8) (169,1.030604833e9) (211,1.992986e9) (253,2.607595792e9) (295,3.393051125e9) (337,4.396471584e9) (379,5.717839958e9) (421,7.101320125e9) (463,8.323566916000001e9) (505,9.866025708e9)
            };
            \addlegendentry{$\tilde{p}_{10^5}$ with \cref{alg:arnoldi-iteration}}
            \addplot coordinates {
                (1,31264000) (51,156812000) (101,473340000) (151,1188000000) (201,2168000000) (251,3771000000) (301,5899000000) (351,7564000000) (401,10155000000) (451,13009000000) (501,16293000000) (551,20174000000)
            };
            \addlegendentry{$\tilde{p}_{10^5}$ with \cref{alg:arnoldi-aggr-crit}}
            \addplot+[no marks, gray, dashed]{1.000360925e10}
                node[pos=0.275, above] {Computing $p_{10^5}$};
        \end{axis}
    \end{tikzpicture}
    \begin{tikzpicture}
        \begin{axis}[
            xlabel = {Dimension of Exlump aggregation},
            ylabel = {Runtime in nanseconds},
            ymode = log,
            xmin = 3,
            xmax = 5523,
            ymin = 5e8,
            ymax = 2e11,
            domain=3:5523,
            legend pos = south east,
            width = 0.95\linewidth,
            height = 6.0cm,
            colorblindstyle,
        ]
            \addplot coordinates {
                (5523,83314050912.85706) (4389,120294332981.10962) (3339,89052643060.6842) (3011,108575690984.72595) (2824,98519396305.08423) (2658,62364818096.16089) (2539,58140552997.58911) (2258,70762601852.41699) (1803,89563687801.36108) (1531,102784945964.81323) (919,35332195997.23816) (493,22250364065.170288) (209,8075666904.449463) (50,5679363012.313843) (27,2677618741.9891357) (3,537408828.7353516)
            };
            \addlegendentry{Exlump}
            \addplot coordinates {
                (5523,88728263854.98047) (4389,124613713026.04675) (3339,91810101985.9314) (3011,111375453948.97461) (2824,101483062982.5592) (2658,64589128971.09985) (2539,60579787969.58923) (2258,73266474962.2345) 
                (1803,91525496006.01196) (1531,104852325201.03455) (919,34730596780.77698) (493,20454180002.212524) (209,7190715074.539185) (50,4480700969.696045) (27,2837113142.01355) (3,774098873.1384277)
            };
            \addlegendentry{$\tilde{p}_{10^5}$ with Exlump}
            \addplot+[no marks, gray, dashed]{81227878808}
                node[pos=0.135, above] {Computing $p_{10^5}$};
        \end{axis}
    \end{tikzpicture}
    \caption{Runtime of computing Arnoldi and Exlump aggregations of the workstation cluster model, both with and without computing $\tilde{p}_{10^5}$, depending on the dimension of the aggregations.}
    \label{fig:runtimeclustercomp}
\end{figure}
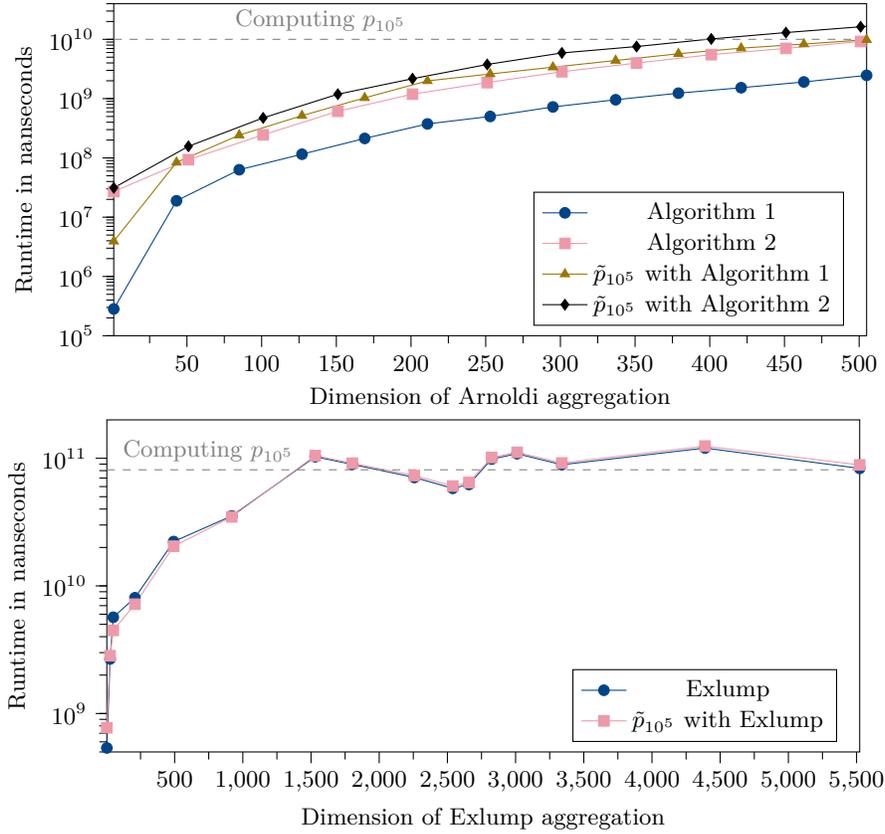
Exlump aggregations fail to consistently beat the naive approach for computing transient distributions, whereas \cref{alg:arnoldi-aggr-crit} is faster up to $j \approx 400$.
An Arnoldi aggregation with $j \approx 300$ has $\norm{\errvec_{10^5}}_1 \approx 10^{-8}$, while being almost four times faster than computing $p_{10^5}$.

Errors for our largest model, the gene expression model, can be seen in \cref{fig:errorgenecomp}.
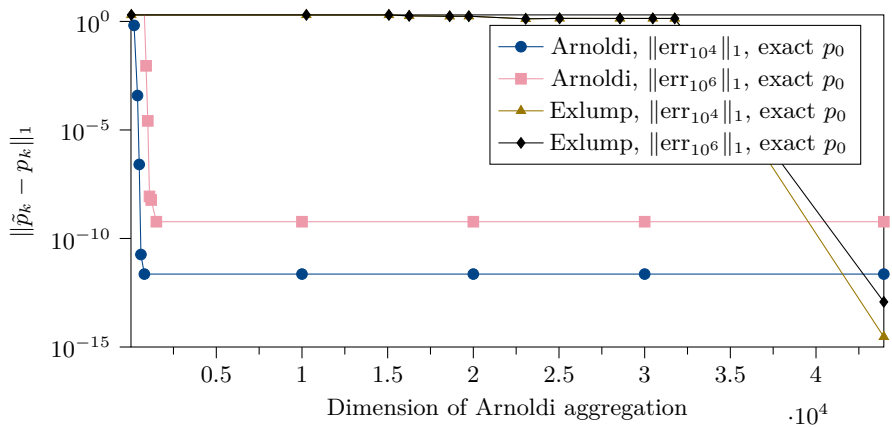
\begin{figure}[ht]
    \centering
    \begin{tikzpicture}
        \begin{axis}[
            xlabel = {Dimension of Arnoldi aggregation},
            ylabel = {$\norm{\tilde{p}_k - p_k}_1$},
            ymode = log,
            xmin = 20,
            xmax = 43957,
            ymin = 1e-15,
            ymax = 2,
            legend pos = north east,
            width = 0.95\linewidth,
            height = 6.0cm,
            colorblindstyle,
        ]
            \addplot coordinates {
                (1,0.9999999999999444) (201,0.6587687932467388) (401,0.0003855514649099467) (501,2.573163693666238e-7) (601,1.8229725718925325e-11) (801,2.290723929352115e-12) (10000,2.29145345565224e-12) (20000,2.29136852853893e-12) (30000,2.29136437438633e-12) (43957,2.29136839814596e-12)
            };
            \addlegendentry{Arnoldi, $\norm{\errvec_{10^4}}_1$, exact $p_0$}
            \addplot coordinates {
                (1,0.99999999999884) (401,130.99441088094216) (801,2.3713831499619022) (901,0.009014472373013217) (1001,2.6196155723189905e-5) (1101,8.694450146679285e-9) (1201,5.926427299167131e-9) (1501,5.926424503754961e-10) (10000,5.926424573529503e-10) (20000,5.926424575234568e-10) (30000,5.926424564017603e-10) (43957,5.9264245769198454e-10)
            };
            \addlegendentry{Arnoldi, $\norm{\errvec_{10^6}}_1$, exact $p_0$}
            \addplot coordinates {
                (9,1.9961241470164677) (42,1.996182001490757) (10257,1.9939769913780414) (15072,1.9931514373096355) (16255,1.803599246729235) (18620,1.7311750713546123) (19738,1.7402383080063302) (23058,1.3188373676417149) (25035,1.3855001493612977) (28557,1.3728751636567667) (30476,1.3761130324038422) (31755,1.3727051844015754) (43957,2.950509031464231e-15)
            };
            \addlegendentry{Exlump, $\norm{\errvec_{10^4}}_1$, exact $p_0$}
            \addplot coordinates {
                (9,1.9961241568073651) (42,1.996182001724781) (10257,1.9939769913942451) (15072,1.99315143733417554) (16255,1.803599246730348) (18620,1.7311750713656678) (19738,1.74023830856783336) (23058,1.3188373678278831) (25035,1.3855001494415672) (28557,1.372875163707443) (30476,1.3761130325532358) (31755,1.3727051844443452) (43957,1.1961529790945572e-13)
            };
            \addlegendentry{Exlump, $\norm{\errvec_{10^6}}_1$, exact $p_0$}
        \end{axis}
    \end{tikzpicture}
    \caption{$\norm{\errvec_{10^4}}_1$ and $\norm{\errvec_{10^6}}_1$ in the gene expression model for Arnoldi and Exlump aggregations, depending on the aggregation dimension.}
    \label{fig:errorgenecomp}
\end{figure}
The only usable Exlump aggregation is the one with $j = n$, thus disqualifying Exlump for this model, while Arnoldi aggregations quickly improve and reach their optimum at around $j \approx 1{,}100$, again limited by rounding errors.

The same holds if we look at the runtime needed to reduce the state space of the gene expression model, as in \cref{fig:runtimegene}.
We did not plot the runtime needed by Exlump aggregations, as they failed to produce any acceptable errors, as seen in \cref{fig:errorgenecomp}.
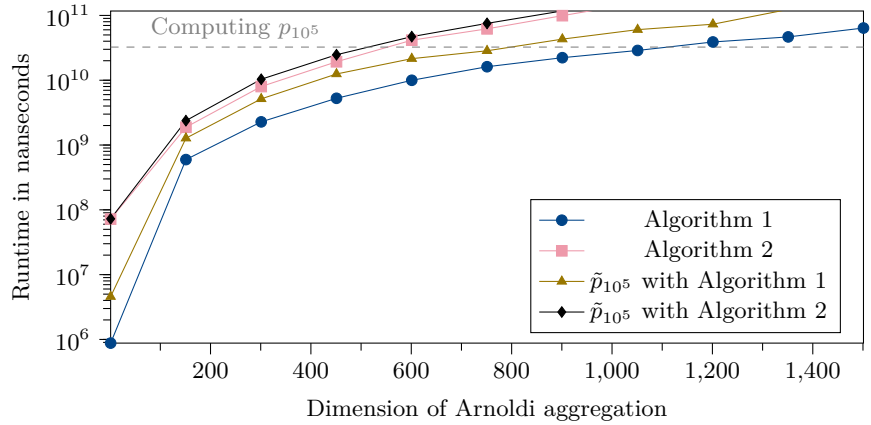
\begin{figure}[ht]
    \centering
    \begin{tikzpicture}
        \begin{axis}[
            xlabel = {Dimension of Arnoldi aggregation},
            ylabel = {Runtime in nanseconds},
            ymode = log,
            xmin = 1,
            xmax = 1501,
            ymin = 877708.0,
            ymax = 1.17e11,
            domain=1:1501,
            legend pos = south east,
            width = 0.95\linewidth,
            height = 6.0cm,
            colorblindstyle,
        ]
            \addplot coordinates {
                (1,877708.0) (151,5.9600975e8) (301,2.276919e9) (451,5.26298625e9) (601,9.984156584e9) (751,1.6156159959e10) (901,2.2274822833e10) (1051,2.8800052708e10) (1201,3.8946550542e10) (1351,4.660551475e10) (1501,6.3921290125e10)
            };
            \addlegendentry{\cref{alg:arnoldi-iteration}}
            \addplot coordinates {
                (1,71995000) (151,1882000000) (301,7992000000) (451,19273000000) (601,41492000000) (751,62089000000) (901,99156000000) (1051,156885000000)
            };
            \addlegendentry{\cref{alg:arnoldi-aggr-crit}}
            \addplot coordinates {
                (1,4.580625e6) (151,1.273896833e9) (301,5.155652917e9) (451,1.2486718792e10) (601,2.1526933875e10) (751,2.8473287625e10) (901,4.3083307875e10) (1051,6.036477925e10) (1201,7.3185328125e10) (1351,1.23776516042e11)
            };
            \addlegendentry{$\tilde{p}_{10^5}$ with \cref{alg:arnoldi-iteration}}
            \addplot coordinates {
                (1,72602000) (151,2377000000) (301,10352000000) (451,24727000000) (601,47043000000) (751,75422000000) (901,117347000000)
            };
            \addlegendentry{$\tilde{p}_{10^5}$ with \cref{alg:arnoldi-aggr-crit}}
            \addplot+[no marks, gray, dashed] {3.2470448625000004e10}
                node[pos=0.167, above] {Computing $p_{10^5}$};
        \end{axis}
    \end{tikzpicture}
    \caption{Runtime of computing Arnoldi aggregations of the gene expression model, both with and without computing $\tilde{p}_{10^5}$, depending on the dimension of the aggregation.}
    \label{fig:runtimegene}
\end{figure}
With such a large model, the dense vector-matrix multiplication forced by a dense $H_j$ in Arnoldi aggregations is no longer dominant.
Instead, the convergence criterion takes up most of the runtime, around 75\%.
$j \approx 450$ is the last dimension of the Arnoldi aggregation where we see a speed-up.
However, an error of $\norm{\errvec_{10^5}}_1 \approx 10^{-1}$ at this size makes it irrelevant for practical use.

Lastly, although not plotted, we get to the Lotka-Volterra model.
It mirrors the behaviour of the gene expression model at a smaller size, as Exlump again fails to find a usable aggregation except for $j = n$, while Arnoldi aggregations converge very soon, although still too late to be fast enough for practical applications.

%% file: 05-conclusion.tex
\section{Conclusion}\label{sec:conclusion}
\paragraph{Results}
\cref{def:arnoldi-aggregation} defines an Arnoldi iteration based aggregation of a \abbrev{dtmc} that incorporates the initial distribution $p_0$.
\cref{thrm:smallest-initial-exact-aggr,thrm:smallest-exact-aggr} then show that, under exact arithmetic, any (initially) exact Arnoldi aggregation is of minimal size.
In practice, however, numerical instability in the Arnoldi iteration (see \cref{subsec:convergence}) prevents exactness.
Instead, we adopt the convergence criterion $\inp*{\abs*{\pi}, \abs*{H_j Q_j - Q_j P} \cdot \mathbf{1}_n} \leq \varepsilon$, although choosing $\varepsilon$ remains heuristic, since no direct bound links $\norm{\errvec_k}_1$ to $\inp*{\abs*{\pi}, \abs*{H_j Q - Q P} \cdot \mathbf{1}_n}$.
Empirically, the relationship in \cref{fig:rsvp-conv-crit} holds across all tested models.

In experiments, small Arnoldi aggregations suffer from dense $H_j$, such that computing $\tilde p_k$ has cost $\mathcal O(kj^2)$, whereas computing $p_k$ costs only $\mathcal O(kn)$ for sparse $P$.
As the model size grows, computing the eigenvector $\pi$ of $H_j$ eventually dominates the cost of the dense vector-matrix product.
Overall, \cref{alg:arnoldi-aggr-crit} runs in $\mathcal{O}(n^2j + nj^2 + j^3)$ for dense $P$ and in $\mathcal O(nj^2 + j^3)$ for sparse ones.
Despite these costs, we observed substantial speed‑ups at low error bounds for the workstation cluster model, outperforming Exlump aggregations in this case.
Moreover, Arnoldi aggregations consistently produce non‑trivial aggregations, whereas Exlump sometimes fails to do so.

\paragraph{Outlook}
We see several avenues for future work.
One is to evaluate our method against additional aggregation techniques (e.g.,~\cite{abate2021aggregation,bittracher2021aggregation,bucholz2014aggregation,michel2025formalbounds}) over a broader suite of models and a wider range of time points.
A second is to speed up Algorithm~\ref{alg:arnoldi-aggr-crit} by dynamically tuning how often the convergence criterion
$\inp*{\abs*{\pi}, \abs*{H_j Q_j - Q_j P} \cdot \mathbf{1}_n} \leq \varepsilon$
is checked.
Expanding on this, when the criterion is first satisfied at dimension $j+\ell$ but not at $j$, one can exploit the recursive structure of $H_{j+\ell}$ and $Q_{j+\ell}$ along with the Krylov-Schur approach of~\cite{stewart2002krylovschur} to identify an intermediate aggregation dimension $j'$ (with $j < j' < j+\ell$) that still meets the criterion at lower computational cost for $\tilde{p}_k$.
Lastly, the Arnoldi iteration itself could be optimized via parallel or randomized Gram-Schmidt variants (e.g.,~\cite{balabanov2022rgs,lingen2000parallelgs}), improved memory management, or integration with high‑performance libraries such as \abbrev{arpack}~\cite{lehoucq1998arpack}.